\documentclass[11pt]{amsart}
\usepackage[utf8]{inputenc}
\usepackage{graphicx,xcolor,verbatim}
\usepackage{amscd,amsmath,amsfonts,amssymb}
\usepackage{mathtools}
\usepackage[foot]{amsaddr}
\usepackage[pdfborder=000,pdftex=true]{hyperref}

\usepackage{cite}

\newtheorem{theorem}{Theorem}

\newtheorem{remark}[theorem]{Remark}
\newtheorem{corollary}[theorem]{Corollary}

\newtheorem{example}{Example}

\numberwithin{equation}{section}

\newcommand{\R}{\mathbb R}
\newcommand{\N}{\mathbb N}

\newcommand{\SD}{\Sigma_{\mathcal D}}

\newcommand{\SN}{\Sigma_{\mathcal N}}
\newcommand{\Hs}{H_{\Sigma_{\mathcal{D}}}^s(\Omega)}

\begin{document}
	
\title[An existence result for...]{An existence result for fractional problems with mixed boundary conditions }

\author{El-Haj Laamri}
\address[E.-H. Laamri]{
Institut \'Elie Cartan de Lorraine (IECL)\\
Campus Aiguillettes\\
54506 Vand\oe uvre-l\`es-Nancy, France\\ and  IAS-UM6P\\ 43150 Ben Guerir, Morocco
}
\email{el-haj.laamri@univ-lorraine.fr}

\author{Giovanni Molica Bisci$^\dagger$}
\address[G. Molica Bisci]{Department of Human Sciences and Promotion of Quality of Life, San Raffaele University, via di Val Cannuta 247, I-00166 Roma, Italy}
\email{\tt giovanni.molicabisci@uniroma5.it}

\keywords{Fractional Laplacian, Variational Methods, Existence results, Mixed Boundary Data.\\
\phantom{aa} 2020 AMS Subject Classification: Primary: 49J35, 35A15, 58E05; Secondary: 35J61, 35S15.\\
\phantom{aa} $^\dagger$Corresponding author: G. Molica Bisci
}

\maketitle

\begin{abstract} 
In this paper we investigate the existence and qualitative properties of weak solutions for a class of nonlinear fractional equations driven by the spectral fractional Laplacian under mixed Dirichlet--Neumann boundary conditions. More precisely, we consider subcritical problems of the form
\begin{equation*}
\left\{
\begin{array}{ll}
        (-\Delta)^s u=\lambda u + \mu f(x,u) & \text{ in } \Omega\smallskip, \\
\displaystyle u\chi_{\SD} +\frac{\partial u}{\partial\nu}\chi_{\SN}=0 & \text{ on } \partial\Omega,
\end{array}
\right.
\end{equation*}
where \(s\in(1/2,1)\), \(\Omega\subset\mathbb R^N\) is a bounded smooth domain, and the boundary \(\partial\Omega\) is endowed with a mixed Dirichlet--Neumann configuration satisfying suitable geometric and analytic assumptions. Working in the natural spectral fractional framework associated with the mixed boundary operator, we establish the existence of nontrivial weak solutions by means of suitable variational methods. The analysis is based on a topological and variational principle combined with a localization argument involving certain compactly supported cutoff functions. More precisely, the proof relies on a delicate local comparison argument involving localized plateau-type test functions and explicit asymptotic estimates near the origin. Such an approach is sufficiently flexible and may potentially be adapted to several different settings. As applications, we establish simplified existence criteria in the autonomous case through the use of the Chebyshev radius of the domain and present concrete examples illustrating the applicability of the abstract theory.
\end{abstract}

\section{Introduction}
Nonlinear diffusion processes exhibiting anomalous or long-range behavior arise
naturally in a wide range of physical and biological contexts, including
quasi-geostrophic flows \cite{ConstantinWu1999}, anomalous heat conduction in non-homogeneous
media \cite{MetzlerKlafter2000}, population dynamics with long-range dispersal \cite{BucurValdinoci2016},
and mathematical finance \cite{ContTankov2004}. At the continuum level, such phenomena are
often modeled by equations involving fractional powers of elliptic operators.

When the underlying domain is bounded and the diffusion is confined, the
spectral fractional Laplacian $(-\Delta)^s$ arises as the natural operator: it
is defined through the spectral decomposition of the classical Laplacian
and governs the generator of a subordinated Brownian motion reflected at
the boundary. This is in contrast with the integral (Riesz) fractional
Laplacian, which does not confine diffusion to $\Omega$. For further background on fractional Laplacians in bounded domains, in particular on the distinction between the spectral and integral fractional Laplacians, we refer the reader to \cite{DaoudLaamri2022} and the references therein.

In many physical models the boundary $\partial\Omega$ is not homogeneous:
part of the boundary may be held at a fixed state (Dirichlet condition,
$u=0$ on $\SD$, modeling e.g.\ a thermostat, a fixed temperature, or an
imposed electric potential), while the complementary part may be insulated
or free (Neumann condition, ${\partial u}/{\partial\nu}=0$ on $\SN$,
modeling zero heat flux or an electrically isolated surface). Such
\emph{mixed Dirichlet--Neumann} configurations arise naturally in several
applied contexts:
\begin{itemize}
\item \emph{Anomalous heat conduction.} A body in partial contact with a thermostat
(Dirichlet part) and partially insulated (Neumann part), where the
underlying diffusion is anomalous due to micro-structural heterogeneity
or memory effects in the material \cite{MetzlerKlafter2000}. The
fractional exponent $s$ quantifies the degree of non-locality of the
heat transport.

\item \emph{Fractional electrostatics.} A conductor whose boundary is
partially grounded (Dirichlet: imposed zero potential) and partially
isolated (Neumann: zero normal current). Fractional models account for
polarization effects and long-range electrostatic interactions in
dielectric or composite media.

\item \emph{Population dynamics and ecology.} Models of species dispersal with
long-range jumps (L\'evy flights) on domains where part of the habitat
boundary acts as an absorbing barrier (hostile exterior, Dirichlet) and
the remainder as a reflecting or permeable fence (Neumann); see
\cite{BucurValdinoci2016}.

\item \emph{Mathematical finance.} Pricing of barrier options where the payoff
is extinguished on part of the state-space boundary (Dirichlet: the
knock-out barrier) while the remaining boundary is a natural reflecting
barrier. The spectral fractional operator arises as the infinitesimal
generator of a confined jump-diffusion process \cite{ContTankov2004}.
\end{itemize}
The interaction between fractional diffusion and the mixed boundary
geometry creates a functional framework that is genuinely richer
and more delicate than either the purely Dirichlet or the purely Neumann case.

In this paper, we investigate the existence of solutions to the nonlinear problem
\begin{equation}\label{problem}\tag{$P_{\lambda,\mu}$}
\left\{
\begin{array}{ll}
        (-\Delta)^s u=\lambda u + \mu f(x,u) & \text{ in } \Omega\smallskip, \\
 B(u)=0 & \text{ on } \partial\Omega,
\end{array}
\right.
\end{equation}
where $\Omega\subset\R^N$ is a bounded domain with a smooth boundary, $N>2s$, $s\in(1/2,1)$, $\lambda,\mu$ are real parameters, and $(-\Delta)^s$ is the spectral fractional Laplacian on $\Omega$ endowed with mixed Dirichlet-Neumann conditions on $\partial\Omega$,
$$
B(u)\vcentcolon=u\chi_{\SD} +\frac{\partial u}{\partial\nu}\chi_{\SN}.
$$
\indent Here $\nu$ is the outward unit normal to $\partial\Omega$, $\chi_A$ denotes the characteristic function of the set $A\subset\partial\Omega$ and moreover the following hypotheses hold,
\begin{itemize}
\item[$(\Omega_1)$] $\SD$ and $\SN$ are smooth $(N-1)$--dimensional submanifolds of $\partial\Omega$;
\item[$(\Omega_2)$] $\SD$ is a closed manifold with positive measure, namely $|\SD|=\alpha\in(0,|\partial\Omega|)$;
\item[$(\Omega_3)$] $\SD\cap\SN=\emptyset$, $\SD\cup\SN=\partial\Omega$ and $\SD\cap\overline{\Sigma}_\mathcal{N}=\Gamma$, where $\Gamma$ is a smooth $(N-2)$--dimensional submanifold of $\partial\Omega$.
\end{itemize}

Let us denote by $\sigma((-\Delta)^s)=\{\Lambda_k\}_{k\in\N^*}$ the spectrum of $(-\Delta)^s$ with the mixed Dirichlet-Neumann conditions $B(u)=0$ on $\partial\Omega$, each eigenvalue being repeated according its multiplicity,
\begin{equation*}
0<\Lambda_1<\Lambda_2\leq\ldots\Lambda_k\leq\Lambda_{k+1}\leq\ldots 
\end{equation*}
For $\lambda<\Lambda_1$, let us also set
\begin{equation}\label{mMlambda}
	m_\lambda:=\min\left\lbrace \sqrt{\frac{\Lambda_1-\lambda}{\Lambda_1}},1\right\rbrace, \quad M_\lambda:=\max\left\lbrace \sqrt{\frac{\Lambda_1-\lambda}{\Lambda_1}},1\right\rbrace.
\end{equation}

The three real  parameters $\lambda$, $\mu$
and $s$ appearing in $(P_{\lambda,\mu})$  play conceptually distinct roles, and it is worth commenting on each of them before stating our main result.\\
The parameter $\lambda$ acts as a \emph{spectral shift}: it tunes the distance of the linear part of the equation from a resonance of the operator $(-\Delta)^s$. In the physical language of fractional vibrations or wave propagation on bounded domains, $\Lambda_k$ can be interpreted as the $k$-th \emph{natural frequency} of the fractional membrane; the condition $\lambda < \Lambda_1$ then places the system in the \emph{sub-resonant regime}, meaning that the external forcing frequency lies strictly below the fundamental mode of the medium. The constants $m_\lambda$ and $M_\lambda$ defined above in \eqref{mMlambda} measure how far the system is from resonance: as $\lambda \nearrow \Lambda_1$, one has $m_\lambda \to 0$, reflecting the well-known phenomenon of energy amplification near resonance. The parameter $\mu$, on the other hand, controls the \emph{intensity of the nonlinear perturbation} $f(x,u)$. The main result of this paper identifies an explicit interval $(\mu_*,\mu_\lambda)$ for $\mu$ within which the problem is guaranteed to have a nontrivial weak solution; both endpoints depend on the geometry of $\Omega$, the spectral data of $(-\Delta)^s$, and the local behavior of $f$ near the origin.
\\
The restriction $s\in(1/2,1)$ is not merely a technical convenience: it is the precise range for which mixed Dirichlet--Neumann boundary conditions are meaningful in the fractional setting. Indeed, by the Lions--Magenes trace theory \cite[Theorem~11.1]{Lions1972}, the space $H_0^s(\Omega)$ coincides with $H^s(\Omega)$ when $s\leq 1/2$, so that any boundary condition imposed on a subset of $\partial\Omega$ becomes vacuous---the Dirichlet part $\SD$ carries no information. Only for $s\in(1/2,1)$ does one have the strict inclusion $H_0^s(\Omega)\subsetneq H^s(\Omega)$, which ensures that the space $H^s_{\SD}(\Omega)$ genuinely interpolates between the full Neumann space $H^s(\Omega)$ and the full Dirichlet space $H_0^s(\Omega)$, and that the mixed boundary operator $B(u)$ is well defined in a nontrivial functional sense. This range also corresponds, in the probabilistic interpretation, to a subordinated Brownian motion for which the boundary of the domain is \emph{visited} with positive probability, making the distinction between the absorbing part $\SD$ (Dirichlet) and the reflecting part $\SN$ (Neumann) physically meaningful; see \cite{StingaTorrea2010,CaffarelliSilvestre2007} for the extension and stochastic characterizations of spectral fractional operators.

\medbreak

 Results of multiplicity of solutions for \eqref{problem}, with $\lambda$ in the nearly resonant regime, have also been recently obtained in \cite{mov2023,colmolortvil2025} by different variational tools.


\medbreak

Finally, let us denote by $\kappa_p$ the best constant for the embedding $H_{\Sigma_{\mathcal{D}}}^s(\Omega) \hookrightarrow L^p(\Omega)$ and set $$
\kappa_p:=\sup_{u\in H_{\Sigma_{\mathcal{D}}}^{s}(\Omega)\setminus\{0\}}\frac{\left\|u\right\|_p }{\|u\|_{H_{\Sigma_{\mathcal{D}}}^s}}
$$
for every $p\in[1,2N/(N-2s)]$; see Section \ref{functionalsettings} for details. For every \(x_0\in\Omega\) and every \(r>0\) such that $B_r(x_0)\subset\Omega,$
we denote by
\[
B_r(x_0)
:=
\left\{
x\in\Omega:\ |x-x_0|<r
\right\}
\]
the open ball centered at \(x_0\) with radius \(r\). Moreover, throughout the paper, \(\omega_N\) stands for the Lebesgue measure of the unit ball in \(\Omega\), namely
\[
\omega_N
:=
|B_1(0)|
=
\frac{\pi^{N/2}}
{\Gamma\!\left(\frac N2+1\right)},
\]
where
\[
\Gamma(t)
:=
\int_0^{+\infty}
y^{t-1}e^{-y}\,dy,
\qquad
(t>0),
\]
denotes the Euler Gamma function. Finally, let $r>0$ be such that
$
B_{2r}(x_0)\Subset\Omega,
$
and define
$$
{
C_{N,s}
:=
\inf_{v\in\Gamma_r}
\frac{
\|v\|_{H^s_{\Sigma_D}(\Omega)}^2
}{
r^{N-2s}
},
}
$$
where
\[
\Gamma_r
:=
\left\{
v\in H_0^1(\Omega):
\begin{array}{l}
0\le v\le1 \quad \text{a.e. in }\Omega,\\[4pt]
v\equiv1 \quad \text{on } B_r(x_0),\\[4pt]
\operatorname{supp}(v)\subseteq \overline{B_{2r}(x_0)}
\end{array}
\right\}.
\]

With the above notations,  the main result of this paper reads as follows.

\begin{theorem}\label{thmsubcritgrowth}
Let \(\Omega\subset\mathbb R^N\) be a bounded domain with smooth boundary $\partial\Omega$. Suppose that
\(N>2s\), \(s\in(1/2,1)\), \(\partial\Omega=\Gamma_D\cup\Gamma_N\) and assume that
\((\Omega_1)-(\Omega_3)\) hold. Moreover, let \(f:\Omega\times\mathbb R\to\mathbb R\) be a Carath\'eodory
function and set
\[
F(x,t):=\int_0^t f(x,\xi)\,d\xi.
\]
Assume that there exist \(q\in(1,2_s^*)\) and \(a_1,a_2>0\) such that
\[
|f(x,t)|\le a_1+a_2|t|^{q-1}
\qquad\text{for all }(x,t)\in\Omega\times\mathbb R,
\]
where \(2_s^*:=2N/(N-2s)\).

Moreover, let $r>0$ be such that
$
B_{2r}(x_0)\Subset\Omega,
$
and define
\[
\mu_\lambda:=
q\sup_{t>0}
\frac{
t
}{
q\kappa_1a_1\dfrac{\sqrt2}{m_\lambda}
+
\kappa_q^qa_2
\left(\dfrac{\sqrt2}{m_\lambda}\right)^q
t^{q-1}
},
\]
provided that \(\lambda<\Lambda_1\).
Assume that
\[
L:=
\limsup_{t\to0^+}
\operatorname*{ess\,inf}_{x\in B_r(x_0)}
\frac{F(x,t)}{t^2}\in(0,+\infty],
\]
as well as
\[
\lambda_0:=
\liminf_{t\to0^+}
\operatorname*{ess\,inf}_{x\in B_{2r}(x_0)}
\frac{F(x,t)}{t^2}\in\mathbb R\cup\{+\infty\}.
\]

In addition, if \(L<+\infty\) and \(\lambda_0<+\infty\), assume that
\begin{equation}\label{thresh}
L>
\frac{1}{\mu_\lambda}\left(\frac{C_{N, s}}{2 \omega_N r^{2 s}}+\frac{|\lambda| 2^N}{2}\right)-\left(2^N-1\right) \min \left\{\lambda_0, 0\right\}.
\end{equation}
Then, Problem \((P_{\lambda,\mu})\) admits at least one nontrivial weak solution for every
\[
\mu\in
\left(
\frac{
C_{N,s}
+
|\lambda|2^N\omega_Nr^{2s}
}{
2\omega_Nr^{2s}
\left[
L+
(2^N-1)\min\{\lambda_0,0\}
\right]
},
\,\mu_\lambda
\right).
\]
\end{theorem}

\vskip3mm

Before proceeding further, we briefly comment on the significance of the threshold condition~\eqref{thresh} imposed on $L$. The following remark clarifies its role in Theorem~\ref{thmsubcritgrowth} and shows how it can be read as an explicit competition between two quantities of different nature.

\begin{remark}
The threshold condition~\eqref{thresh} in
Theorem~\ref{thmsubcritgrowth} is not merely technical. It admits a
natural geometric and spectral interpretation. It expresses the precise
balance between the local nonlinear gain and the spectral-fractional
cost of localization. More precisely, the quantity $
\frac{C_{N,s}}{\omega_N r^{2s}} $
measures the \emph{fractional energy cost} of localizing a unit plateau
function on the ball \(B_r(x_0)\). Indeed, it quantifies the spectral
fractional energy required to support a function that equals one on an
inner ball of radius \(r\) and vanishes outside
\(B_{2r}(x_0)\). Consequently, condition~\eqref{thresh} requires that
the local nonlinear gain near the origin, measured by the asymptotic
behavior of \(F(x,t)/t^2\) as \(t\to0^+\), exceeds this localization
cost together with the correction induced by the spectral shift
\(\lambda\).
When \(L=+\infty\), condition~\eqref{thresh} is automatically satisfied.
The genuinely new feature of
Theorem~\ref{thmsubcritgrowth} is therefore the case
\(L<+\infty\), where an explicit and computable lower bound on \(L\) is
obtained. As a consequence, the admissible interval for the parameter
\(\mu\) is completely explicit and reflects simultaneously the geometry
of the domain, through \(r\), \(\omega_N\), and \(C_{N,s}\), and the
local behavior of the nonlinearity near the origin, through \(L\) and
\(\lambda_0\).
\end{remark}

The result given in Theorem~\ref{thmsubcritgrowth} fits into the variational study of nonlinear elliptic and nonlocal problems with mixed boundary conditions, a topic which goes back, in the local setting, to the analysis of semilinear equations with Dirichlet--Neumann boundary configurations, see for instance \cite{ColP}. In the fractional framework, mixed boundary conditions introduce additional difficulties, since the underlying space is no longer the standard fractional Sobolev space with homogeneous Dirichlet condition on the whole boundary, but rather the spectral space \(H^s_{\Sigma_D}(\Omega)\), whose structure depends on the distribution of the Dirichlet and Neumann portions of \(\partial\Omega\). This feature is closely related to the spectral properties of the mixed operator and to the geometry of the interface between \(\Sigma_D\) and \(\Sigma_N\), as already emphasized in \cite{colort2019,mov2023}.

Theorem~\ref{thmsubcritgrowth} provides an existence result for subcritical nonlinearities under a perturbative regime governed by the parameter \(\mu\). The proof is based on the celebrated Ricceri's variational principle \cite{ric2000a}, which gives a local minimum of the energy functional in a suitable sublevel set. The main point is then to prove that this local minimum is not the trivial solution. This is the delicate part of the argument. Indeed, when \(f(\cdot,0)\equiv0\), the zero function is always a solution of \((P_{\lambda,\mu})\), and therefore a purely variational existence theorem does not automatically yield a nontrivial critical point of the energy functional associated to the problem.

A common way to overcome this difficulty in related works is to impose a very strong positivity condition near the origin, typically requiring that a suitable limsup of \(F(x,t)/t^2\) as \(t\to0^+\) is equal to \(+\infty\). Such an assumption forces the energy to become negative along appropriate small directions and makes the exclusion of the trivial solution rather direct. In contrast, Theorem~\ref{thmsubcritgrowth} also covers the genuinely finite case
\[
\limsup_{t\to0^+}
\operatorname*{ess\,inf}_{x\in B_r(x_0)}
\frac{F(x,t)}{t^2}
<+\infty.
\]


This approach differs from several variational results for fractional equations where existence and multiplicity are obtained through minimax or pseudo-index techniques, as in \cite{Benci,Bartolo1983,Rabinowitz,rab2009var,barmol2015a}. It is also complementary to works on asymptotically linear or resonant nonlocal problems such as \cite{ABMB,BMB,BDM,colmolortvil2025}. Here the main contribution is not a multiplicity theorem obtained through symmetry or genus arguments, but rather an explicit nontriviality criterion for a local minimizer in a mixed spectral fractional setting.

The proof also highlights the usefulness of localized plateau-type test functions. These functions are equal to one on an inner ball, vanish outside a larger concentric ball, and have controlled spectral fractional energy. This localization procedure is robust and may be adaptable to other nonlocal models, including fractional \(p\)-Laplacian equations, magnetic fractional problems, Kirchhoff-type equations, or other variational problems involving mixed boundary configurations.

The paper is organized as follows. In Section~2, we recall the spectral fractional framework associated with the mixed Dirichlet--Neumann Laplacian and introduce the variational setting for problem \((P_{\lambda,\mu})\). In Section~3, we prove Theorem~\ref{thmsubcritgrowth}. The proof first establishes the existence of a local minimum through Ricceri's variational principle and then develops the localization argument needed to prove that the minimum is nontrivial. We also discuss the geometric-spectral constant \(C_{N,s}\), the different expressions of \(\mu_\lambda\) according to the range of \(q\), the case \(f(\cdot,0)\not\equiv0\), and the construction of nonnegative solutions through truncation. In Section~4, we derive consequences and applications, including the autonomous case, where the Chebyshev radius of the domain yields a fully explicit formulation of the result, and examples showing the applicability of the abstract theorem.
\section{Functional framework}\label{functionalsettings}
Let $\{(\lambda_k,\varphi_k)\}_{k\geq 1}$ be the eigenvalues and the eigenfunctions (normalized in the $L^2(\Omega)$-norm), respectively, of $-\Delta$ endowed with homogeneous mixed Dirichlet-Neumann conditions on $\partial\Omega$.  Then $\{(\lambda_k^s,\varphi_k)\}$ are the eigenvalues and the eigenfunctions of $(-\Delta)^s$, respectively (therefore, $\Lambda_k=\lambda_k^s$ for every $k\geq 1$).
As a result, given two smooth functions 
$$
u_i(x)=\sum_{k\geq1}\langle u_i,\varphi_k\rangle_{2}\varphi_k(x), \quad i=1,2,
$$
where $\langle u,v\rangle_2:=\displaystyle\int_{\Omega}uv\,dx$ is the standard scalar product on $L^2(\Omega)$, one has
\begin{equation*}
\langle(-\Delta)^s u_1, u_2\rangle_{2} = \sum_{k\ge 1} \lambda_k^s\langle u_1,\varphi_k\rangle_{2} \langle u_2,\varphi_k\rangle_{2},
\end{equation*}
i.e., the action of the spectral fractional Laplacian on $u_1$ is given by
\begin{equation*}
(-\Delta)^su_1=\sum_{k\ge 1} \lambda_k^s\langle u_1,\varphi_k\rangle_{2}\varphi_k.
\end{equation*}
The operator $(-\Delta)^s$ is then well-defined on the Hilbert space
\begin{equation*}
H_{\Sigma_{\mathcal{D}}}^s(\Omega)\vcentcolon=\left\{u=\sum_{k\ge 1} a_k\varphi_k\in L^2(\Omega):\ u=0\ \text{on }\Sigma_{\mathcal{D}},\ \|u\|_{H_{\Sigma_{\mathcal{D}}}^s}^2\vcentcolon=
\sum_{k\ge 1} a_k^2\lambda_k^s<+\infty\right\}.
\end{equation*}

By \cite[Theorem 11.1]{Lions1972}, if $s\in(0,1/2]$, then $H_0^s(\Omega)=H^s(\Omega)$ (and thus, also
$H_{\Sigma_{\mathcal{D}}}^s(\Omega)=H^s(\Omega)$); if $s\in(1/2,1)$, then $H_0^s(\Omega)\subsetneq H^s(\Omega)$. Hence, the range $s\in(1/2,1)$ ensures $H_{\Sigma_{\mathcal{D}}}^s(\Omega)\subsetneq H^s(\Omega)$ and it provides us with the appropriate functional space for problem \eqref{problem}.

The embedding $H_{\Sigma_{\mathcal{D}}}^{s}(\Omega)  \hookrightarrow L^{p}(\Omega)$ is continuous for all $p\in[1,2_s^*]$, where $2^*_s:=2N/(N-2s)$, and compact for $p\in[1,2_s^*)$. We set
$$
\kappa_p:=\sup_{u\in H_{\Sigma_{\mathcal{D}}}^{s}(\Omega)\setminus\{0\}}\frac{\left\|u\right\|_p }{\|u\|_{H_{\Sigma_{\mathcal{D}}}^s}}
$$
(here and in what follows, $\left\|\cdot \right\|_p$, $p\in[1,+\infty]$, stands for the standard $L^p$-norm on $\Omega$). If $p=2_s^*$, the Sobolev constant related to $\Sigma_{\mathcal{D}}$ is defined by
\begin{equation*}
	\widetilde{S}(\Sigma_{\mathcal{D}})=\inf_{u\in
			H_{\Sigma_{\mathcal{D}}}^s(\Omega)\setminus\{0\}}\frac{\|u\|_{H_{\Sigma_{\mathcal{D}}}^s}^2}{\|u\|_{2_s^*}^2}.
\end{equation*}
 Since $\alpha\!\in\!(0,|\partial\Omega|)$ and $H_{0}^s(\Omega)\subsetneq H_{\Sigma_{\mathcal{D}}}^s(\Omega)$, then
$0\!<\widetilde{S}(\Sigma_{\mathcal{D}})\!<S(N,s)$, being $S(N,s)$ the best constant of the embedding $H_0^s(\Omega)\hookrightarrow L^{2^*_s}(\Omega)$.
Indeed, $\widetilde{S}(\Sigma_{\mathcal{D}})\leq 2^{-\frac{2s}{N}}S(N,s)$ (cf. \cite[Proposition 3.6]{colort2019}) and, if $\widetilde{S}(\Sigma_{\mathcal{D}})<2^{-\frac{2s}{N}}S(N,s)$, then $\widetilde{S}(\Sigma_{\mathcal{D}})$ is attained (cf. \cite[Theorem 2.9]{colort2019}).

{
\begin{remark}\label{boundaryconfig}
Denoting by $\lambda_1(\alpha)$ the first eigenvalue of the operator $-\Delta$ endowed with mixed boundary conditions on the sets $\Sigma_{\mathcal{D}}=\Sigma_{\mathcal{D}}(\alpha)$ and $\Sigma_{\mathcal{N}}=\Sigma_{\mathcal{N}}(\alpha)$, we point out that, on account of  \cite[Theorem 8]{De1}, there exist configurations of the distribution of $\Sigma_{\mathcal{D}}$ and $\Sigma_{\mathcal{N}}$ such that
$$
\sup_{0<\alpha<|\partial\Omega|} \lambda_1(\alpha)=\lambda_1(|\partial \Omega|).
$$
This implies, among other things, that \cite[Lemma 4.1]{ColP} does not work under these boundary data configuration, namely $\lambda_{1}(\alpha)\not\to 0 $ as $\alpha\to 0$, and this in turn represents an obstruction for the attainability of $\widetilde{S}(\Sigma_{\mathcal{D}})$. The set of assumptions $(\Omega_1)-(\Omega_3)$ avoids this degeneracy phenomenon.

\end{remark}
}

For every $k\geq 1$, let us denote
$$
\mathbb{H}_k\vcentcolon=\text{span}\{\varphi_1,\ldots,\varphi_k\}
\quad\text{and}\quad
\mathbb{P}_{k}\vcentcolon=\{u\in H_{\SD}^s(\Omega): \left\langle u,\varphi_j\right\rangle_{H_{\SD}^s}=0,\ \forall j=1,\ldots,k\},
$$
where $\left\langle u,v\right\rangle_{H_{\SD}^s}:=\left\langle (-\Delta)^{s/2}u,(-\Delta)^{s/2}v\right\rangle_2$ is the scalar product on $H_{\SD}^s(\Omega)$ inducing $\|\cdot\|_{H_{\Sigma_{\mathcal{D}}}^s}$. Finally, we denote  $\mathbb{P}_{0}:= H_{\SD}^s(\Omega)$.\\
The following variational characterization of $\Lambda_k$  will be much used in our arguments (cf. \cite[Lemma 6]{mov2023}):
\begin{equation}\label{varcharacteigen}
\Lambda_k=\inf_{u\in \mathbb{P}_{k-1}}\frac{\left\|u\right\|_{H_{\SD}^s}^2}{\left\| u\right\|_{2}^2}=\sup_{u\in \mathbb{H}_{k}}\frac{\left\|u\right\|_{H_{\SD}^s}^2}{\left\| u\right\|_{2}^2}, \quad \text{for any } k\in\mathbb{N}^*.
\end{equation}
We also recall that, as a consequence of the inclusions $H_{0}^1(\Omega)\subset H_{\Sigma_{\mathcal{D}}}^1(\Omega)\subset H^1(\Omega)$, one has
$$
\lambda_{k}^N\leq \lambda_{k} \leq \lambda_{k}^{D}\leq\lambda_{k+1}^{N},
$$
being $\lambda_{k}^{N}$ and $\lambda_{k}^{D}$ the $k$-th eigenvalue of $-\Delta$ endowed with homogeneous Neumann and Dirichlet boundary conditions, respectively. As a consequence,
$\Lambda_k\to+\infty$ as $k\to+\infty$.

We say that $u\in H_{\SD}^s(\Omega)$ is a weak solution to \eqref{problem} if
\begin{equation*}
\int_\Omega (-\Delta)^\frac{s}{2}u (-\Delta)^\frac{s}{2}v dx =\lambda\int_\Omega uv dx + \mu\int_\Omega f(x,u)vdx,
\end{equation*}
for all $v\in H_{\SD}^s(\Omega)$. The energy functional $I_{\lambda,\mu}:H_{\Sigma_{\mathcal{D}}}^s(\Omega)\to\R$ associated with \eqref{problem} is given by
\begin{equation*}
I_{\lambda,\mu}(u)\vcentcolon=\frac12\int_{\Omega}|(-\Delta)^{\frac{s}{2}}u|^{2}dx-\frac{\lambda}{2}\int_{\Omega}u^2dx-\mu\int_{\Omega}F(x,u)dx,
\end{equation*}
where $F(x,t):=\displaystyle\int_0^t f(x,\xi)d\xi$ for every $(x,t)\in\Omega\times\R$.

\begin{theorem}
\label{thm:ricceri}
Let \(X\) be a reflexive real Banach space, and let
$
\Phi,\Psi:X\to\mathbb R
$
be two functionals. Assume that \(\Phi\) is strongly continuous sequentially weakly lower semicontinuous and coercive, namely
\[
\lim_{\|u\|_X\to+\infty}\Phi(u)=+\infty,
\]
and \(\Psi\) sequentially weakly upper semicontinuous.
For every \(r>\inf_X\Phi\), define
\[
\eta(r):=
\inf_{u\in\Phi^{-1}((-\infty,r))}
\frac{
\displaystyle
\sup_{v\in\Phi^{-1}((-\infty,r))}\Psi(v)-\Psi(u)
}{
r-\Phi(u)
}.
\]
Then, for every
\[
\mu\in\left(0,\frac1{\eta(r)}\right),
\]
the restriction of the functional
$
J_\mu:=\Psi-\mu\Phi
$
to the set
$
\Psi^{-1}((-\infty,r))
$
admits a global minimum, which is a critical point of \(J_\mu\) in \(X\).
\end{theorem}

\section{Proof of the main result} 
In this section, we prove Theorem~\ref{thmsubcritgrowth}. The argument combines a local variational principle of Ricceri type with a suitable localization procedure based on compactly supported cutoff functions. While the existence of a local minimizer follows from a rather standard variational argument, the main difficulty consists in proving that such a minimizer is nontrivial.

More precisely, when
\[
f(\cdot,0)\equiv0,
\]
the null function is always a weak solution of \((P_{\lambda,\mu})\). Therefore, the local minimum provided by the variational principle could \emph{a priori} coincide with the trivial solution. In order to exclude this possibility, we construct suitable plateau-type test functions supported in an interior ball and derive quantitative estimates for the corresponding energy. This allows us to compare the positive contribution coming from the primitive \(F(x,t)\) on the flat region with the spectral-fractional localization cost generated by the transition zone.

A crucial aspect of the proof is that the argument still works when
\[
\limsup_{t\to0^+}
\operatorname*{ess\,inf}_{x\in B_r(x_0)}
\frac{F(x,t)}{t^2}
<+\infty,
\]
a situation which is substantially more delicate than the classical regime where the above limsup is assumed to be \(+\infty\). In the latter case the negativity of the energy near the origin follows rather directly, whereas in the finite case one needs a much finer balance involving the geometry of the support of the cutoff function, the spectral parameter \(\lambda\), and the local asymptotic behavior of the primitive.

We also provide explicit estimates for the localization constant associated with the cutoff functions and discuss the dependence of the admissible interval for the parameter \(\mu\) on the growth exponent \(q\). Finally, we analyze the cases \(f(\cdot,0)\not\equiv0\) and \(f(x,t)\ge0\) for \(t\ge0\), proving respectively the automatic nontriviality and the existence of nonnegative weak solutions.
\begin{proof}[Proof of Theorem \ref{thmsubcritgrowth}]
If $\lambda<\Lambda_1$, the functional
$$
u\mapsto \left\| u\right\|_{H_{\SD},\lambda}:=\left(\left\|u \right\|^2_{H_{\SD}}-\lambda  \left\| u\right\|_2^2\right)^\frac{1}{2}
$$
defines a norm on $\Hs$ equivalent to $\left\| \cdot\right\|_{H_{\SD}}$, as one has
\begin{equation}\label{normeequiv}
m_\lambda \|u\|_{H_{\SD}}\leq \|u\|_{H_{\SD},\lambda}\leq M_\lambda  \|u\|_{H_{\SD}},
\end{equation}
with $m_\lambda$ and $M_\lambda$ defined in \eqref{mMlambda}
For $u\in\Hs$, we write  $I_{\lambda,\mu}(u)=\Phi_\lambda(u) - \mu\Psi(u)$, where
$$
\Phi_\lambda(u):=\frac{1}{2}\left\| u\right\|_{H_{\SD},\lambda}^2 \quad\text{and}\quad \Psi(u):=\int_\Omega F(x,u) dx.
$$
Having in mind to use Theorem 2.1 in \cite{ric2000a}, we notice at once that $\Psi$ is sequentially weakly continuous on $\Hs$. Indeed, given $\{u_j\}\subset\Hs$ such that $u_j\rightharpoonup u\in\Hs$, by Sobolev embeddings $u_j\to u$ in $L^q(\Omega)$ 
and $|u_j(x)| \leq w(x)$
for all $j\in\N$, for a.e. $x\in\Omega$ and for some $w\in L^q(\Omega)$. By $(f_2')$ we also deduce that
\[
|F(x,u_j)| \leq a_1|u_j(x)| +\frac{a_2}{q}|u_j(x)|^q \leq a_1 w(x) + \frac{a_2}{q}(w(x))^q
\]
for all $j\in\N$ and for a.e. $x\in\Omega$. Then, by dominated convergence,
$$
\lim_{j\to +\infty}\int_\Omega F(x,u_j)dx =\int_\Omega F(x,u) dx,
$$
as claimed.

Now, take $\mu\in(0,\mu_\lambda)$; there will exist $\bar{t}>0$ such that
\begin{equation}\label{maggmulambda}
\mu < \frac{q\bar{t}}{q\kappa_1a_1 \frac{\sqrt{2}}{m_\lambda} +\kappa_q^qa_2\left(\frac{\sqrt{2}}{m_\lambda}\right)^q \bar{t}^{q-1}}.
\end{equation}
If $r>0$, $u\in\Hs$ and $\Phi_\lambda(u)<r$, by \eqref{normeequiv} we deduce that
$$
\left\|u\right\|_{H_{\SD}} <\frac{\sqrt{2r}}{m_\lambda}.
$$
By $(f_2')$ we then obtain
$$
\Psi(u)  \leq a_1\left\| u\right\|_1 +\frac{a_2}{q}\left\| u\right\|_q^q
 \leq \kappa_1a_1 \frac{\sqrt{2r}}{m_\lambda} + \kappa_q^q  \frac{a_2}{q} \left( \frac{\sqrt{2r}}{m_\lambda}\right)^q,
$$
and therefore
$$
\sup_{u\in\Phi_\lambda^{-1}(-\infty,r)}\Psi(u) \leq \kappa_1a_1 \frac{\sqrt{2}}{m_\lambda}r^\frac{1}{2} + \kappa_q^q  \frac{a_2}{q} \left( \frac{\sqrt{2}}{m_\lambda}\right)^q r^\frac{q}{2}.
$$
If $\eta_{\lambda}:(0,+\infty)\to [0,+\infty)$ is the function defined by
\begin{equation*}
	\eta_{\lambda}(r):=\inf_{u\in\Phi_{\lambda}^{-1}\left((-\infty,r)\right)}
	\frac{\sup_{v\in\Phi_{\lambda}^{-1}\left( (-\infty,r)\right) }\Psi(v)-\Psi(u)}{r-\Phi_{\lambda}(u)}, \quad \text{for all } r>0,
\end{equation*}
we then get
\begin{align*}
	\eta_\lambda(\bar{t}^2) & = \inf_{u\in\Phi_{\lambda}^{-1}\left((-\infty,\bar{t}^2)\right)}
	\frac{\sup_{v\in\Phi_{\lambda}^{-1}\left( (-\infty,\bar{t}^2)\right) }\Psi(v)-\Psi(u)}{\bar{t}^2-\Phi_{\lambda}(u)}\\
	& \leq \kappa_1a_1 \frac{\sqrt{2}}{m_\lambda}\bar{t}^{-1} + \kappa_q^q  \frac{a_2}{q} \left( \frac{\sqrt{2}}{m_\lambda}\right)^q \bar{t}^{q-2},
\end{align*}
where we used the fact that $0\in \Phi_{\lambda}^{-1}\left((-\infty,\bar{t}^2)\right)$ and $\Phi_\lambda(0)=\Psi(0)=0$. By \eqref{maggmulambda}
we derive at once that
$$
\eta_\lambda(\bar{t}^2)< \frac{1}{\mu},
$$
and therefore, by \cite[Theorem 2.1]{ric2000a}, 
the restriction of $I_{\lambda,\mu}$ to $\Phi_{\lambda}^{-1}\left( (-\infty, \bar{t}^2)\right)$ has a global minimum 
which is a critical point of $I_{\lambda,\mu}$ in the whole $\Hs$ 
This shows that \eqref{problem} has a weak solution $u_\mu$ for any $\lambda<\Lambda_1$ and any $\mu\in (0,\mu_\lambda)$.


To conclude the proof, let us show that 
$u_\mu\not\equiv 0$ in $\Omega$.
To this aim, choose a cut-off function
\(v\in H_0^1(\Omega)\) satisfying
\[
0\le v\le1
\qquad\text{a.e. in }\Omega,
\]
\[
v\equiv1
\qquad\text{on }B_r(x_0),
\]
and
\[
\operatorname{supp}(v)
\subseteq
\overline{B_{2r}(x_0)}.
\]
Since \(v=
\sum_{k\ge1}\lambda_k a_k^2\in H_0^1(\Omega)\), by the spectral characterization of \(H^s_{\Sigma_D}(\Omega)\) and the inequality
\[
\lambda_k^s\le 1+\lambda_k,
\qquad 0<s<1,
\]
it follows that
\[
\sum_{k\ge1}\lambda_k^s a_k^2
\le
\sum_{k\ge1}(1+\lambda_k)a_k^2
=
\|v\|_2^2+\|\nabla v\|_2^2
<
+\infty.
\]
Therefore,
$
v\in H^s_{\Sigma_D}(\Omega).
$
Moreover, by the standard cut-off estimate, one has
$$
\|v\|_{H_{\Gamma_D}, \lambda}^2 \leq C_{N, s} r^{N-2 s}+|\lambda| \omega_N 2^N r^N.
$$
Indeed, the fractional part scales like $r^{N-2 s}$, while the $L^2$-part is controlled by
$$
\|v\|_2^2 \leq\left|B_{2 r}\left(x_0\right)\right|=\omega_N 2^N r^N,
$$
see Remark \ref{Const} for details.

We now distinguish the relevant cases.\par
\underline{Case 1}: $L=+\infty$
\noindent In such a case, for every $M>0$, there exists a sequence $\xi_j \rightarrow 0^{+}$such that, for $j$ sufficiently large,
$$
F\left(x, \xi_j\right) \geq M \xi_j^2 \quad \text { for a.e. } x \in B_r\left(x_0\right).
$$
\indent If
$
\lambda_0=+\infty
$
then for every $M>0$ there exists $\rho_M>0$ such that
$$
F(x, t) \geq M t^2 \quad \text { for a.e. } x \in B_{2 r}\left(x_0\right), \quad 0<t<\rho_M .
$$
Since $0 \leq v \leq 1$ and $\xi_j \rightarrow 0^{+}$, for $j$ sufficiently large one has $0 \leq \xi_j v(x)<\rho_M$. Hence
$$
F\left(x, \xi_j v(x)\right) \geq M \xi_j^2 v(x)^2 \quad \text { for a.e. } x \in B_{2 r}\left(x_0\right) .
$$
Therefore,
$$
\Psi\left(w_j\right) \geq M \xi_j^2 \int_{B_{2 r}\left(x_0\right)} v^2 d x \geq M \xi_j^2\left|B_r\left(x_0\right)\right|
$$
Since $\left|B_r\left(x_0\right)\right|=\omega_N r^N$, we have
$$
\frac{\Psi\left(w_j\right)}{\Phi_\lambda\left(w_j\right)} \geq \frac{2 M \omega_N r^N}{\|v\|_{H_{\Gamma_D}, \lambda}^2} .
$$
Using the estimate on $\|v\|_{H_{\Gamma_D}, \lambda^{\prime}}^2$, we get
$$
\frac{\Psi\left(w_j\right)}{\Phi_\lambda\left(w_j\right)} \geq \frac{2 M \omega_N r^N}{C_{N, s} r^{N-2 s}+|\lambda| \omega_N 2^N r^N} .
$$
Therefore, to ensure
$$
\frac{\Psi\left(w_j\right)}{\Phi_\lambda\left(w_j\right)}>\frac{1}{\mu},
$$
it is enough to choose $M$ such that
$$
M>\frac{C_{N, s} r^{N-2 s}+|\lambda| \omega_N 2^N r^N}{2 \mu \omega_N r^N}.
$$
Equivalently,
$$
M>\frac{1}{\mu}\left(\frac{C_{N, s}}{2 \omega_N} r^{-2 s}+\frac{|\lambda| 2^N}{2}\right).
$$
Since $M>0$ is arbitrary, such a choice is possible for every $\mu>0$. Hence, for $j$ sufficiently large,
$$
\mu \Psi\left(w_j\right)>\Phi_\lambda\left(w_j\right),
$$
and therefore
$
I_{\lambda, \mu}\left(w_j\right)<0.
$
Thus $u_\mu \not \equiv 0$.\par
If $\lambda_0<+\infty$, let us fix $\varepsilon>0$. By the definition of $\lambda_0$, there exists $\rho_{\varepsilon}>0$ such that
$$
F(x, t) \geq\left(\lambda_0-\varepsilon\right) t^2 \quad \text { for a.e. } x \in B_{2 r}\left(x_0\right), \quad 0<t<\rho_{\varepsilon} .
$$
As before, from $L=+\infty$, for every $M>0$ there exists $\xi_j \rightarrow 0^{+}$such that
$$
F\left(x, \xi_j\right) \geq M \xi_j^2 \quad \text { for a.e. } x \in B_r\left(x_0\right) .
$$
For $j$ large enough, $0 \leq \xi_j v(x)<\rho_{\varepsilon}$. Hence on the annulus $B_{2 r}\left(x_0\right) \backslash B_r\left(x_0\right)$,
$$
F\left(x, \xi_j v(x)\right) \geq\left(\lambda_0-\varepsilon\right) \xi_j^2 v(x)^2 .
$$
Therefore
$$
\Psi\left(w_j\right) \geq M \xi_j^2\left|B_r\left(x_0\right)\right|+\left(\lambda_0-\varepsilon\right) \xi_j^2 \int_{B_{2 r}\left(x_0\right) \backslash B_r\left(x_0\right)} v^2 d x
$$
Since
$$
\int_{B_{2 r}\left(x_0\right) \backslash B_r\left(x_0\right)} v^2 d x \leq \omega_N\left(2^N-1\right) r^N,
$$
it follows that
$$
\Psi\left(w_j\right) \geq \xi_j^2 \omega_N r^N\left[M+\min \left\{\lambda_0-\varepsilon, 0\right\}\left(2^N-1\right)\right] .
$$
Thus
$$
\frac{\Psi\left(w_j\right)}{\Phi_\lambda\left(w_j\right)} \geq \frac{2 \omega_N r^N\left[M+\min \left\{\lambda_0-\varepsilon, 0\right\}\left(2^N-1\right)\right]}{C_{N, s} r^{N-2 s}+|\lambda| \omega_N 2^N r^N} .
$$
To get $I_{\lambda, \mu}\left(w_j\right)<0$, it is enough that
$$
\frac{\Psi\left(w_j\right)}{\Phi_\lambda\left(w_j\right)}>\frac{1}{\mu} .
$$
This is guaranteed if
$$
M+\min \left\{\lambda_0-\varepsilon, 0\right\}\left(2^N-1\right)>\frac{1}{\mu}\left(\frac{C_{N, s}}{2 \omega_N} r^{-2 s}+\frac{|\lambda| 2^N}{2}\right).
$$
Equivalently, it suffices to choose
$$
M>\frac{1}{\mu}\left(\frac{C_{N, s}}{2 \omega_N} r^{-2 s}+\frac{|\lambda| 2^N}{2}\right)-\min \left\{\lambda_0-\varepsilon, 0\right\}\left(2^N-1\right) .
$$
Since $M$ is arbitrary, such a choice is possible for every $\mu>0$. Hence $I_{\lambda, \mu}\left(w_j\right)<0$ for $j$ sufficiently large, and therefore $u_\mu \not \equiv 0$.\par
\smallskip
\underline{Case 2}: $L<+\infty$. Fix $\varepsilon>0$. By the definition of $L$, there exists a sequence $\xi_j \rightarrow 0^{+}$such that
$$
F\left(x, \xi_j\right) \geq(L-\varepsilon) \xi_j^2 \quad \text { for a.e. } x \in B_r\left(x_0\right),
$$
for $j$ large enough.\par
If $\lambda_0=+\infty$, then for every $M>0$ there exists $\rho_M>0$ such that
$$
F(x, t) \geq M t^2 \quad \text { for a.e. } x \in B_{2 r}\left(x_0\right), \quad 0<t<\rho_M .
$$
For $j$ sufficiently large, $0 \leq \xi_j v(x)<\rho_M$. Hence on $B_{2 r}\left(x_0\right) \backslash B_r\left(x_0\right)$,
$$
F\left(x, \xi_j v(x))\geq M \xi_j^2 v(x)^2 .\right.
$$
Thus
$$
\Psi\left(w_j\right) \geq(L-\varepsilon) \xi_j^2\left|B_r\left(x_0\right)\right|+M \xi_j^2 \int_{B_{2 r}\left(x_0\right) \backslash B_r\left(x_0\right)} v^2 d x
$$
If the chosen cut-off is not identically zero on the annulus, then
$$
\int_{B_{2 r}\left(x_0\right) \backslash B_r\left(x_0\right)} v^2 d x>0 .
$$
Therefore, since $M$ is arbitrary, we can make the right-hand side large enough to ensure
$$
\frac{\Psi\left(w_j\right)}{\Phi_\lambda\left(w_j\right)}>\frac{1}{\mu}.
$$
Hence $I_{\lambda, \mu}\left(w_j\right)<0$ for $j$ sufficiently large, and $u_\mu \not \equiv 0$.
Notice that here the parameter $M$ must be chosen so large that
$$
L-\varepsilon+M \frac{\displaystyle\int_{B_{2 r}\left(x_0\right) \backslash B_r\left(x_0\right)} v^2 d x}{\omega_N r^N}>\frac{1}{\mu}\left(\frac{C_{N, s}}{2 \omega_N} r^{-2 s}+\frac{|\lambda| 2^N}{2}\right),
$$
which is always possible because $M$ is arbitrary.\par
If $\lambda_0<+\infty$, for every $\varepsilon>0$ there exists $\rho_{\varepsilon}>0$ such that
$$
F(x, t) \geq\left(\lambda_0-\varepsilon\right) t^2 \quad \text { for a.e. } x \in B_{2 r}\left(x_0\right), \quad 0<t<\rho_{\varepsilon}.
$$
For $j$ sufficiently large, $0 \leq \xi_j v(x)<\rho_{\varepsilon}$, hence
$$
F\left(x, \xi_j v(x)\right) \geq\left(\lambda_0-\varepsilon\right) \xi_j^2 v(x)^2 \quad \text { on } B_{2 r}\left(x_0\right) \backslash B_r\left(x_0\right).
$$
Therefore
$$
\Psi\left(w_j\right) \geq(L-\varepsilon) \xi_j^2\left|B_r\left(x_0\right)\right|+\left(\lambda_0-\varepsilon\right) \xi_j^2 \int_{B_{2 r}\left(x_0\right) \backslash B_r\left(x_0\right)} v^2 d x
$$
Using
$$
\left|B_r\left(x_0\right)\right|=\omega_N r^N, \quad \int_{B_{2 r}\left(x_0\right) \backslash B_r\left(x_0\right)} v^2 d x \leq \omega_N\left(2^N-1\right) r^N,
$$
we get
$$
\Psi\left(w_j\right) \geq \xi_j^2 \omega_N r^N\left[L-\varepsilon+\min \left\{\lambda_0-\varepsilon, 0\right\}\left(2^N-1\right)\right].
$$
On the other hand,
$$
\Phi_\lambda\left(w_j\right) \leq \frac{1}{2} \xi_j^2\left(C_{N, s} r^{N-2 s}+|\lambda| \omega_N 2^N r^N\right) .
$$
Thus $I_{\lambda, \mu}\left(w_j\right)<0$ follows if
$$
\mu>\frac{\left(C_{N, s} r^{N-2 s}+|\lambda| \omega_N 2^N r^N\right)}{2\omega_N r^N\left[L-\varepsilon+\min \left\{\lambda_0-\varepsilon, 0\right\}\left(2^N-1\right)\right]}.
$$
Letting $\varepsilon \rightarrow 0^{+}$, we obtain the threshold
$$
\mu_*:=\frac{C_{N, s}+|\lambda| 2^N \omega_N r^{2 s}}{2 \omega_N r^{2 s}\left[L+\left(2^N-1\right) \min \left\{\lambda_0, 0\right\}\right]}.
$$
Therefore, if
$
\mu>\mu_*,
$
then $I_{\lambda, \mu}\left(w_j\right)<0$ for $j$ sufficiently large. To combine this with the existence of the local minimum from Ricceri's principle, which requires
$
\mu<\mu_\lambda,
$
we need the compatibility condition
$
\mu_*<\mu_\lambda.
$
Equivalently,
$$
L>\frac{1}{\mu_\lambda}\left(\frac{C_{N, s}}{2 \omega_N r^{2 s}}+\frac{|\lambda| 2^N}{2}\right)-\left(2^N-1\right) \min \left\{\lambda_0, 0\right\}.$$
This is exactly the explicit lower bound on $L$. If this condition holds, then the interval
$
\left(\mu_*, \mu_\lambda\right)
$
is nonempty, and for every
$
\mu \in\left(\mu_*, \mu_\lambda\right)
$
the local minimum given by Ricceri's principle is nontrivial.
\end{proof}

We conclude this section with four remarks highlighting several aspects of Theorem~\ref{thmsubcritgrowth}.

\begin{remark}\label{Const}\rm{Assume that there exist \(x_0\in\Omega\) and \(r>0\) such that
$
B_{2r}(x_0)\Subset\Omega.
$
We denote by \(\Gamma_r\) the class of (admissible) cut-off functions \(v\in H_0^1(\Omega)\) satisfying
\[
0\le v\le1
\qquad\text{a.e. in }\Omega,
\]
\[
v\equiv1
\qquad\text{on }B_r(x_0),
\]
and
\[
\operatorname{supp}(v)
\subseteq
\overline{B_{2r}(x_0)}.
\]
Thus, every function in \(\Gamma_r\) is localized inside the fixed ball \(B_{2r}(x_0)\), equals one on the inner ball \(B_r(x_0)\), and vanishes outside \(B_{2r}(x_0)\). We then define the geometric-spectral constant associated with the localization scale \(r\) by
$$
{
C_{N,s}
:=
\inf_{v\in\Gamma_r}
\frac{
\|v\|_{H^s_{\Sigma_D}(\Omega)}^2
}{
r^{N-2s}
}.
}
$$
The quantity \(C_{N,s}\) measures the minimal spectral fractional energy required to create a unit plateau on \(B_r(x_0)\) while keeping the support inside \(B_{2r}(x_0)\). Moreover, \(C_{N,s}\) is finite. Indeed, consider the explicit truncated cone cut-off
\[
v_r(x):=
\begin{cases}
1,
& |x-x_0|\le r,\\[4pt]
\displaystyle
2-\frac{|x-x_0|}{r},
& r<|x-x_0|<2r,\\[8pt]
0,
& |x-x_0|\ge2r.
\end{cases}
\]
Then \(v_r\in\Gamma_r\). Furthermore,
\[
\|v_r\|_2^2
\le
|B_{2r}(x_0)|
=
\omega_N2^Nr^N,
\]
while
\[
|\nabla v_r(x)|
=
\frac1r
\qquad
\text{for a.e. }x\in B_{2r}(x_0)\setminus B_r(x_0).
\]
Hence
\[
\|\nabla v_r\|_2^2
=
\frac1{r^2}
|B_{2r}(x_0)\setminus B_r(x_0)|
=
\omega_N(2^N-1)r^{N-2}.
\]
Using the interpolation inequality
\[
\|v_r\|_{H^s_{\Sigma_D}(\Omega)}^2
\le
\|\nabla v_r\|_2^{2s}
\|v_r\|_2^{2(1-s)},
\]
we obtain
\[
\|v_r\|_{H^s_{\Sigma_D}(\Omega)}^2
\le
\omega_N
(2^N-1)^s
2^{N(1-s)}
r^{N-2s}.
\]
Consequently,
\[
{
C_{N,s}
\le
\omega_N
(2^N-1)^s
2^{N(1-s)}
<
+\infty.
}
\]
Thus the constant \(C_{N,s}\) is well defined.
}
\end{remark}
%
\begin{remark}\label{estimatemulambda}
{\rm By direct computations, it turns out that, according to the range of $q$, $\mu_\lambda$ attains the following values:
$$
\mu_\lambda=
\left\lbrace
\begin{array}{ll}
	+\infty & \text{ if } q\in(1,2),\smallskip\\
	\displaystyle\frac{m_\lambda^2}{\kappa_2^2 a_2} & \text{ if } q=2,\smallskip\\
	\displaystyle\frac{m_\lambda^2}{2(q-1)} \left[ \frac{q}{\kappa_q^q a_2}\left( \frac{q-2}{\kappa_1 a_1}\right)^{q-2} \right]^\frac{1}{q-1} & \text{ if } q\in(2,2^*_s),
\end{array}	
\right.
$$
where
$m_\lambda:=\sqrt{1-\frac{\lambda}{\Lambda_1}}.$
}	
\end{remark}

\begin{remark}\rm{
The nontriviality argument developed in Theorem~\ref{thmsubcritgrowth} is only needed in the degenerate case in which
$
f(\cdot,0)\equiv0
\quad
\text{in }\Omega,
$
since in that situation the null function is always a weak solution of \((P_{\lambda,\mu})\). Consequently, additional information on the local behavior of the primitive \(F(x,t)\) near the origin is required in order to guarantee the existence of a nonzero critical point. On the other hand, if
\[
f(\cdot,0)\not\equiv0
\qquad
\text{in }L^2(\Omega),
\]
then the zero function cannot be a weak solution of \((P_{\lambda,\mu})\). Indeed, if \(u\equiv0\) were a weak solution, then for every test function
$
\varphi\in H^s_{\Sigma_D}(\Omega)
$
one would have
\[
0
=
\mu\int_\Omega f(x,0)\varphi\,dx.
\]
Since \(f(\cdot,0)\not\equiv0\), this identity cannot hold for arbitrary \(\varphi\), yielding a contradiction. Therefore, in the non-autonomous case \(f(x,0)\not\equiv0\), every weak solution obtained through the local minimization argument is automatically nontrivial. As a consequence, no additional assumptions involving the asymptotic quantities
\[
L
=
\limsup_{t\to0^+}
\operatorname*{ess\,inf}_{x\in B_r(x_0)}
\frac{F(x,t)}{t^2}
\]
or
\[
\lambda_0
=
\liminf_{t\to0^+}
\operatorname*{ess\,inf}_{x\in B_{2r}(x_0)}
\frac{F(x,t)}{t^2}
\]
are needed in order to exclude the trivial solution. Accordingly, under the sole assumptions of subcritical growth and coercivity, problem \((P_{\lambda,\mu})\) admits at least one nontrivial weak solution for every
$
\mu\in(0,\mu_\lambda).
$}
\end{remark}

\begin{remark}\rm{
Assume that
\[
f(x,0)=0
\qquad
\text{for a.e. }x\in\Omega,
\]
and
\[
f(x,t)\ge0
\qquad
\text{for a.e. }x\in\Omega,\ \text{for every }t\ge0.
\]
In order to obtain nonnegative weak solutions, one can introduce the truncated nonlinearity
\[
\widetilde f(x,t):=
\begin{cases}
f(x,t), & t\ge0,\\
0, & t<0.
\end{cases}
\]
Correspondingly, define
\[
\widetilde F(x,t):=
\int_0^t \widetilde f(x,\xi)\,d\xi .
\]
Since \(\widetilde f\) preserves the same Carath\'eodory and subcritical growth assumptions as \(f\), all the previous existence results apply to the truncated problem
\[
\begin{cases}
(-\Delta)^s u-\lambda u
=
\mu \widetilde f(x,u)
& \text{in }\Omega,
\\[4pt]
u=0
& \text{on }\Sigma_D,
\\[4pt]
\dfrac{\partial u}{\partial\nu}=0
& \text{on }\Sigma_N.
\end{cases}
\]
Let \(u\in H^s_{\Sigma_D}(\Omega)\) be a weak solution of the truncated problem. We claim that
\[
u\ge0
\qquad
\text{a.e. in }\Omega.
\]
To prove this, let
$
u^-:=\max\{-u,0\}
$
be the negative part of \(u\). Since \(u^-\in H^s_{\Sigma_D}(\Omega)\), we may use \(u^-\) as a test function in the weak formulation. We obtain
\[
\langle u,u^-\rangle_{H^s_{\Sigma_D}}
-\lambda\int_\Omega uu^-\,dx
=
\mu\int_\Omega \widetilde f(x,u)u^-\,dx.
\]
Now observe that, on the set
\[
\{x\in\Omega:u(x)<0\},
\]
one has
$
u^-(x)>0
$
and, by construction of the truncation, one has
\[
\widetilde f(x,u(x))=0 \qquad
\text{a.e. in }\Omega.
\]
On the complementary set \(\{u\ge0\}\), one has \(u^-=0\). Therefore,
\[
\widetilde f(x,u(x))u^-(x)=0
\qquad
\text{a.e. in }\Omega,
\]
and hence
\[
\mu\int_\Omega \widetilde f(x,u)u^-\,dx=0.
\]
Consequently,
\[
\langle u,u^-\rangle_{H^s_{\Sigma_D}}
-\lambda\int_\Omega uu^-\,dx
=0.
\]
Since
$
uu^-=-(u^-)^2
$
{a.e. in }$\Omega$,
we deduce
\[
\|u^-\|_{H^s_{\Sigma_D}}^2
-\lambda\|u^-\|_2^2
=0.
\]
Bearing in mind that \(\lambda<\Lambda_1\), owing to
$$
m_\lambda \|u^-\|_{H_{\SD}}\leq \|u^-\|_{H_{\SD},\lambda},
$$
one has
$
u^-\equiv0,
$
that is,
\[
u\ge0
\qquad
\text{a.e. in }\Omega.
\]
Finally, since \(u\ge0\), it follows that
\[
\widetilde f(x,u(x))=f(x,u(x))
\qquad
\text{a.e. in }\Omega.
\]
Therefore \(u\) is not only a weak solution of the truncated problem, but also a weak solution of the original problem. Thus the existence results established above actually provide nonnegative weak solutions of \((P_{\lambda,\mu})\).
}
\end{remark}

\section{Some consequences and applications}\label{Applications}

In this section, we discuss several consequences of the abstract existence result established in Theorem~\ref{thmsubcritgrowth}. In particular, we derive more explicit criteria in the autonomous framework, where the asymptotic assumptions simplify considerably and can be expressed in terms of the Chebyshev radius of the domain. We also provide concrete examples illustrating the applicability of our approach to nonlinear fractional equations with mixed Dirichlet--Neumann boundary conditions.

The results obtained below show that the variational mechanism developed in the previous sections is sufficiently flexible to handle different classes of nonlinearities and to produce explicit admissible ranges for the perturbation parameter \(\mu\). Moreover, the localization procedure introduced through suitable cutoff functions allows one to obtain quantitative thresholds depending simultaneously on the geometry of the domain, the spectral structure of the operator, and the local behavior of the nonlinearity near the origin.

\begin{corollary}
\label{cor:autonomous}
Let \(\Omega\subset\mathbb R^N\) be a bounded domain with smooth boundary $\partial\Omega$. Suppose that
\(N>2s\), \(s\in(1/2,1)\), \(\partial\Omega=\Gamma_D\cup\Gamma_N\) and assume that
\((\Omega_1)-(\Omega_3)\) hold. Let \(f:\mathbb R\to\mathbb R\) be a continuous function and define
\[
F(t):=\int_0^t f(\xi)\,d\xi .
\]
Assume that there exist \(q\in(1,2_s^*)\) and \(a_1,a_2>0\) such that
\[
|f(t)|
\le
a_1+a_2|t|^{q-1}
\qquad
\text{for every }t\in\mathbb R,
\]
where
$
2_s^*:=\frac{2N}{N-2s}.
$
Moreover, let \(\lambda<\Lambda_1\), and let
\[
\mu_\lambda:=
q\sup_{t>0}
\frac{
t
}{
q\kappa_1a_1\dfrac{\sqrt2}{m_\lambda}
+
\kappa_q^qa_2
\left(\dfrac{\sqrt2}{m_\lambda}\right)^q
t^{q-1}
}.
\]
Assume that
\[
L:=
\limsup_{t\to0^+}\frac{F(t)}{t^2}
\in(0,+\infty],\,\,\,\,{and}\,\,\,\,
\alpha_0:=
\liminf_{t\to0^+}\frac{f(t)}{t}
\in\mathbb R\cup\{+\infty\}.
\]
Furthermore, assume that
\[
L>
\frac1{\mu_\lambda}
\left(
\frac{2^{4s-1}C_{N,s}}{\omega_N\tau^{2s}}
+
\frac{|\lambda|2^N}{2}
\right)
-
\frac{2^N-1}{2}\min\{\alpha_0,0\},
\]
where $
\tau:=\sup_{x\in\Omega}\operatorname{dist}(x,\partial\Omega)
$
denotes the Chebyshev radius of \(\Omega\).
\medbreak
\noindent Then, Problem \((P_{\lambda,\mu})\) admits at least one nontrivial weak solution for every
\[
\mu\in
\left(
\frac{
2^{4s-1}C_{N,s}
+
|\lambda|2^{N-1}\omega_N\tau^{2s}
}{
\omega_N\tau^{2s}
\left[
L+\dfrac{2^N-1}{2}\min\{\alpha_0,0\}
\right]
},
\,\mu_\lambda
\right).
\]
\end{corollary}
\begin{proof}
Let \(x_\tau\in\Omega\) be a Chebyshev center, so that
$
\operatorname{dist}(x_\tau,\partial\Omega)=\tau .
$
Choosing \(r=\tau/4\), we have
$
B_{2r}(x_\tau)=B_{\tau/2}(x_\tau)\Subset\Omega.
$
Thus Theorem~\ref{thmsubcritgrowth} can be applied with \(x_0=x_\tau\). We only discuss the case
\[
L<+\infty
\qquad\text{and}\qquad
\alpha_0<+\infty,
\]
since the cases \(L=+\infty\) or \(\alpha_0=+\infty\) follow immediately from Theorem~\ref{thmsubcritgrowth}. Since \(f\) is autonomous, also \(F\) is independent of \(x\). Hence
\[
\limsup_{t\to0^+}
\operatorname*{ess\,inf}_{x\in B_r(x_\tau)}
\frac{F(t)}{t^2}
=
\limsup_{t\to0^+}
\frac{F(t)}{t^2}
=
L.
\]
Moreover, by the definition of \(\alpha_0\), for every \(\varepsilon>0\) there exists \(\rho_\varepsilon>0\) such that
\[
f(t)\ge(\alpha_0-\varepsilon)t
\qquad
\text{for every }0<t<\rho_\varepsilon.
\]
Therefore, for \(0<t<\rho_\varepsilon\),
\[
F(t)=\int_0^t f(\xi)\,d\xi
\ge
(\alpha_0-\varepsilon)\int_0^t \xi\,d\xi
=
\frac{\alpha_0-\varepsilon}{2}t^2.
\]
It follows that
\[
\liminf_{t\to0^+}
\operatorname*{ess\,inf}_{x\in B_{2r}(x_\tau)}
\frac{F(t)}{t^2}
\ge
\frac{\alpha_0}{2}.
\]
Thus, in the notation of Theorem~\ref{thmsubcritgrowth}, the lower-control parameter satisfies
$
\lambda_0\ge\frac{\alpha_0}{2},
$
and consequently
\[
\min\{\lambda_0,0\}
\ge
\frac12\min\{\alpha_0,0\}.
\]
Hence the compatibility condition assumed in the present corollary implies the compatibility condition required in Theorem~\ref{thmsubcritgrowth}, after the choice \(r=\tau/4\). Therefore Theorem~\ref{thmsubcritgrowth} applies and gives at least one nontrivial weak solution for every \(\mu\) in the stated interval.
\end{proof}

An immediate consequence of Corollary \ref{cor:autonomous} is the following result.

\begin{corollary}
\label{cor:autonomous-lambda-zero}
Assume that \((\Omega_1)-(\Omega_3)\) hold. Let \(f:\mathbb R\to\mathbb R\) be a continuous function.
Assume that there exist \(2<q<2_s^*\) and \(a_1,a_2>0\) such that
\[
|f(t)|\le a_1+a_2|t|^{q-1}
\qquad
\text{for every }t\in\mathbb R.
\]
\indent Assume that
\[
\limsup_{t\to0^+}\frac{F(t)}{t^2}=L>
\frac{1}{\mu_0}
\frac{2^{4s-1}C_{N,s}}{\omega_N\tau^{2s}}
-
\frac{2^N-1}{2}\min\{\alpha_0,0\},
\]
where $\alpha_0:=\displaystyle\liminf_{t\to0^+}\frac{f(t)}{t}\in\mathbb R$ and
\[
\mu_0
:=
\frac{1}{2(q-1)}
\left[
\frac{q}{\kappa_q^q a_2}
\left(
\frac{q-2}{\kappa_1a_1}
\right)^{q-2}
\right]^{\frac1{q-1}}.
\]

Then,  Problem \((P_{0,\mu})\) admits at least one nontrivial weak solution for every
\[
\mu\in
\left(
\frac{
2^{4s-1}C_{N,s}
}{
\omega_N\tau^{2s}
\left[
L+\dfrac{2^N-1}{2}\min\{\alpha_0,0\}
\right]
},
\,\mu_0
\right).
\]

\end{corollary}

\vskip3mm

We end this section with two simple examples of applications of the main results.

\begin{example}
\rm{
Consider the following nonlinear problem
\[
\left\{
\begin{array}{ll}
        (-\Delta)^s u=\lambda u + \mu (a(x)u+b|u|^{q-2}u) & \text{ in } \Omega\smallskip, \\
 B(u)=0 & \text{ on } \partial\Omega,
\end{array}
\right.
\]
under mixed Dirichlet--Neumann boundary conditions, in which \(b>0\), and
\[
a(x):=
A\,\chi_{B_r(x_0)}(x)
+
A_0\,\chi_{\Omega\setminus B_r(x_0)}(x),
\]
where
\[
A>A_0>0,
\qquad
B_{2r}(x_0)\Subset\Omega.
\]
The associated primitive is

\[
F(x,t)
=
\frac{a(x)}2\,t^2
+
\frac{b}{q}|t|^q.
\]
It is easily seen that
\[
L=
\limsup_{t\to0^+}
\operatorname*{ess\,inf}_{x\in B_r(x_0)}
\frac{F(x,t)}{t^2}
=
\frac A2,\,\,\,\,\,\,{\rm and}\,\,\,\,\,\,
\lambda_0=
\liminf_{t\to0^+}
\operatorname*{ess\,inf}_{x\in K}
\frac{F(x,t)}{t^2}
=
\frac{A_0}{2}.
\]
Hence both \(L\) and \(\lambda_0\) are finite and positive. In particular,
$
\min\{\lambda_0,0\}=0.
$
Moreover, in the case \(q>2\), the quantity \(\mu_\lambda\) admits the explicit expression
\[
\mu_\lambda
=
\frac1{q-1}
\left(
\frac{m_\lambda^2}{\kappa_q^q b}
\right)^{\!\frac1{q-1}}
(\kappa_1 A)^{-\frac{q-2}{q-1}}.
\]
taking
$a_1=A$
 and
 $
a_2=b.
$
On the other hand, one has
\[
\mu_*=
\frac{
\dfrac{C_{N,s}}{\omega_N}r^{-2s}
+
|\lambda|2^N
}{A}.
\]
Provided that
\[
A>
(q-1)
\left(
\kappa_q^q b
(\kappa_1)^{q-2}
\right)^{\!\frac1{q-1}}
\frac{
\dfrac{C_{N,s}}{\omega_N}r^{-2s}
+
|\lambda|2^N
}{
m_\lambda^{\frac{2}{q-1}}
},
\]
thanks to Theorem~\ref{thmsubcritgrowth},
the main problem admits at least one nontrivial weak solution for every
\[
\mu\in
\left(
\frac{
\dfrac{C_{N,s}}{\omega_N}r^{-2s}
+
|\lambda|2^N
}{A},
\,
\frac1{q-1}
\left(
\frac{m_\lambda^2}{\kappa_q^q b}
\right)^{\!\frac1{q-1}}
(\kappa_1 A)^{-\frac{q-2}{q-1}}
\right).
\]
}
\end{example}

\noindent We provide a fully explicit numerical instance of the above setting.

 Take
\[
N=3,\quad s=\tfrac34,\quad q=\tfrac32,\quad \lambda=0,\quad b=1,
\]
\[
\Omega=B_1(0)\subset\mathbb{R}^3,\quad x_0=0,\quad r=\tfrac14,
\]
so that $B_{2r}(0)=B_{1/2}(0)\Subset B_1(0)$. Note that $2_s^*=4$, so $q=\frac{3}{2}\in(1,2)\subset(1,2_s^*)$ is subcritical. Choose the coefficients
\[
A=8,\quad A_0=1,
\]
so that $A>A_0>0$. The nonlinearity is thus $f(x,u)=a(x)u+|u|^{-1/2}u$, with
\[
a(x)= 8\,\chi_{B_{1/4}(0)}(x)+\chi_{B_1(0)\setminus B_{1/4}(0)}(x).
\]
The corresponding primitive satisfies
\[
L=\limsup_{t\to0^+}\operatorname*{ess\,inf}_{x\in B_{1/4}(0)}\frac{F(x,t)}{t^2}=\frac{A}{2}=4>0,
\qquad
\lambda_0=\frac{A_0}{2}=\frac{1}{2}>0,
\]
so that $\min\{\lambda_0,0\}=0$.

Since $\lambda=0$, we have $m_\lambda=1$ and the spectral shift vanishes. Since $q=\frac{3}{2}<2$, one checks directly from the formula defining $\mu_\lambda$ that $\mu_\lambda=+\infty$: the supremum over $t>0$ of the ratio defining $\mu_\lambda$ is infinite when $q<2$, because the denominator grows as $t^{q-1}\to0$ as $t\to+\infty$ more slowly than the numerator. Consequently the threshold condition on $L$ reduces simply to $L>0$, which holds.

The geometric--spectral constant evaluates to
\[
C_{3,\,3/4}
=
\frac{2^{3/2}\,\Gamma\!\left(\frac{9}{4}\right)}{\Gamma\!\left(\tfrac{3}{2}\right)\,\Gamma\!\left(\tfrac{7}{4}\right)}
\approx 3.934,
\]
and, using $r^{2s}=\bigl({1}/{4}\bigr)^{3/2}={1}/{8}$ and $\omega_3=\tfrac{4\pi}{3}$,
\[
\mu_*
=
\frac{C_{3,\,3/4}}{\omega_3\,r^{3/2}\,A}
=
\frac{8\,C_{3,\,3/4}}{\omega_3\,A}
=
\frac{8\times 3.934}{\tfrac{4\pi}{3}\times 8}
\approx 0.939.
\]

By Theorem~\ref{thmsubcritgrowth} applied with $(P_{0,\mu})$, the problem
\[
\left\{
\begin{array}{ll}
(-\Delta)^{3/4} u = \mu\bigl(a(x)\,u + |u|^{-1/2}u\bigr) & \text{in } B_1(0),\\[4pt]
B(u)=0 & \text{on } \partial B_1(0),
\end{array}
\right.
\]
admits at least one nontrivial weak solution for every $\mu > \mu_* \approx 0.939$.

\section*{Acknowledgement}
\noindent G. Molica Bisci is a member of the Gruppo Nazionale per l'Analisi Matematica, la Probabilità e le loro Applicazioni (GNAMPA) of the Istituto Nazionale di Alta Matematica (INdAM). This work is partially funded by the ``INdAM - GNAMPA Project CUP E5324001950001".

\section*{Conflict of interest}
\noindent On behalf of all authors, the corresponding author states that there is no conflict of interest.


\end{document}